\documentclass[12pt]{article} 

\usepackage[T2A]{fontenc} 
\usepackage[utf8]{inputenc} 
\usepackage{comment, amsthm, amssymb, amsmath, mathrsfs, bbm, xcolor,enumitem}  
\usepackage{centernot}
\usepackage{mathtools}
\usepackage{float}
\usepackage[english]{babel}
\newtheorem{theorem}{Theorem} 
\newtheorem{lemma}{Lemma}[section]
\newtheorem{corollary}{Corollary}[section]

\newtheorem{rem}{Remark}[theorem]

\theoremstyle{definition}
\theoremstyle{plain}
\newtheorem{definition}{Definition}[section]

\renewcommand{\phi}{\varphi}

\newcommand{\eps}{\varepsilon}
\newcommand{\He}{\mathbb{H}}

\newcommand{\e}{{\rm e}}
\newcommand{\I}{{\rm i}}

\newcommand{\W}{{\rm W}}

\begin{document}

\title{Quantitative sub-ballisticity of self-avoiding walk on the hexagonal lattice}

\author{Dmitrii Krachun\footnotemark[1]\footnote{Princeton University, {dk9781@princeton.edu}}, \, Christoforos Panagiotis\footnotemark[2]\footnote{University of Bath, {cp2324@bath.ac.uk}}}

\maketitle

\begin{abstract}
    We prove quantitative sub-ballisticity for the self-avoiding walk on the hexagonal lattice. Namely, we show that with high probability a self-avoiding walk of length $n$ does not exit a ball of radius $O(n/\log{n})$. Previously, only a non-quantitative $o(n)$ bound was known from the work of Duminil-Copin and Hammond \cite{DCH13}. As an important ingredient of the proof we show that at criticality the partition function of bridges of height $T$ decays polynomially fast to $0$ as $T$ tends to infinity, which we believe to be of independent interest.
\end{abstract}

\section{Introduction}

Self-avoiding walk, originally introduced in 1953 by Flory \cite{Flory53} as a simple model for studying polymer molecules, quickly attracted the interest of both mathematicians and physicists. Initially, Flory's primary interest laid in understanding the typical characteristics of self-avoiding walk, and in particular the mean squared displacement of its endpoint. Since those early days, a substantial amount of research has been dedicated to exploring the typical geometric properties of this model. Despite its seemingly simple definition, the self-avoiding walk has proven to be challenging to study, with many fundamental questions about its behavior remaining unresolved. For a comprehensive introduction the reader can consult \cite{LecturesSAW,MaSaSAWBook}.

Denote by $c_n$ the number of self-avoiding (i.e.\ visiting every vertex at most once) walks of length $n$ on the hexagonal lattice $\mathbb{H}$
started from some fixed vertex,
e.g.\ the origin. A fundamental quantity in the study of self-avoiding walk is its \textit{connective constant} $$\mu=\lim_{n\to \infty} c_n^{1/n},$$
where the limit is known to exist by a standard subadditivity argument \cite{HM54}. In a breakthrough paper, Duminil-Copin and Smirnov \cite{DCS12} used the \textit{parafermionic observable} for self-avoiding walk to prove that $\mu=\sqrt{2+\sqrt{2}}$ on the hexagonal lattice. Except for simple cases\footnote{There are lattices whose connective constant can be deduced from the connected constant of $\mathbb{H}$, e.g.\ the so-called $3.12^2$ lattice \cite{JG98}.}, this is the only lattice for which the value of $\mu$ is known explicitly. 

Beyond the case of the hexagonal lattice, the model has been studied extensively on the hypercubic lattice $\mathbb{Z}^d$, where it is now well-understood in dimensions $d\geq 5$ by the seminal work of Hara and Slade \cite{HaSla1992,HaSlaCri}. The low-dimensional cases are more challenging, and the gap between what is known and what is conjectured to be true is very large. See \cite{HammersleyWelsh,KestenI,KestenII,DCH13,DCKY14,SAWEnds,H18, DCGHM20} for some of the most important results in low dimensions.

Denote by $\mathbb{P}_n$ the uniform measure on self-avoiding walks on the hexagonal lattice $\mathbb{H}$ of length $n$ starting from the origin. Our main result establishes the first quantitative bound on the displacement of self-avoiding walk.
\begin{theorem}\label{thm: quantitative sub-ballisticity}
There exists $c>0$ such that 
\[
\mathbb{P}_n\left[\max\{\lVert \gamma_k \rVert: 0\leq k\leq n\}\geq 10^{11}\cdot \frac{n}{\log(n+1)}\right]\leq \exp\left(-cn^{2/3}\right)
\]
for every $n\geq 1$.
\end{theorem}

Let us mention that the result of Theorem \ref{thm: quantitative sub-ballisticity} is not sharp. 
The prediction, first made by Flory \cite{Flory53}, asserts that the typical displacement of a self-avoiding walk should scale as $n^{3/4+o(1)}$. This was later deduced under the assumption of conformal invariance of the scaling limit of self-avoiding walk by Lawler, Schramm and Werner in \cite{LSW02}.

An important tool in the proof of Theorem~\ref{thm: quantitative sub-ballisticity} is a family of walks that are called \textit{bridges} -- see Definition~\ref{def: bridges}. These walks were introduced by Hammersley and Welsh \cite{HammersleyWelsh}, where they showed that any self-avoiding walk admits a decomposition into bridges, and used this decomposition to show that 
\[
c_n\leq e^{C\sqrt{n}} \mu^n
\]
for every $n \geq 1$, where $C>0$ is an explicit constant. More recently, Duminil-Copin and Hammond \cite{DCH13} showed that self-avoiding walk on $\mathbb{Z}^d$ is sub-ballistic in any dimension $d\geq 2$, using the natural renewal structure of bridges that allows to construct a probability measure on infinite bridges. However, their result relies on the ergodicity of the latter measure, and for this reason it does not provide any quantitative bound on the displacement of self-avoiding walk. A major obstacle in that direction seems to be the lack of a quantitative bound on the partition function of bridges of a given height at the critical point $x_c=1/\mu$ that goes to $0$ as the height goes to infinity. 

In the next theorem, we obtain a polynomial bound for the partition function of bridges of height $T$ at criticality in the case of the hexagonal lattice (see Definition~\ref{def: bridges}), which we then use to prove Theorem~\ref{thm: quantitative sub-ballisticity}. We remark that the proof of the latter implication is based on techniques different than those of \cite{DCH13}, and it relies on the summability of the half-plane boundary two-point function -- see Lemma~\ref{lem: G sum}. The latter is only known in the case of the hexagonal lattice among low dimensional lattices and its proof uses the parafermionic observable.

\begin{theorem}\label{thm: polynomial decay}
Let $\varepsilon=10^{-10}$. For every $T\geq 1$ we have
\[
B_T\leq 100 \cdot T^{-\varepsilon}.
\]
\end{theorem}

The above result improves on the recent work of Glazman and Manolescu \cite{GM20}, where it was shown that $B_T$ converges to $0$ as $T$ tends to infinity, and furthermore that $B_T < (\log T)^{-1/3}$ along a subsequence. The lower bound $B_T\geq c/T$ was obtained in \cite{DCS12} using the parafermionic observable. It is expected that $B_T\sim T^{-1/4}$ by the results of Lawler, Schramm and Werner \cite{LSW02} on the conjectural scaling limit of self-avoiding walk in dimension $2$. 

The proof of Theorem~\ref{thm: polynomial decay} uses the parafermionic observable, which satisfies a local relation reminiscent of the Cauchy-Riemann equations. This property can be used to obtain non-trivial relations between the partition functions of certain types of walks. For instance, it enables to show that $B_T$ is decreasing as a function of $T$, as well as to bound $B_T$ by $D_T$, the partition function of triangle walks, i.e.\ walks that live in an equilateral triangle and end on either its right or its left side -- see Definition~\ref{def: triangle} and Lemma~\ref{lem: B-D ineq}. As it was observed by Glazman and Manolescu \cite{GM20} (see also \cite{DCGHM20}), the latter partition function is more convenient for certain geometric constructions that involve redirecting the endpoint of a walk, and they used it to concatenate three triangle walks to obtain a $U$-walk, i.e.\ a walk in the upper half-plane connecting two boundary points -- see Definition~\ref{def: U walk} and Figure~\ref{fig: 3 triangles}. 

Our strategy relies on the following two observations. First, one can use the observable to obtain a reverse inequality, i.e.\ bound the partition function of $U$-walks ending at distance $k\geq T$ from the origin in terms of $D_T$. This allows to interpret the result of \cite{GM20} as a recurrence inequality for $D_T$. Second, in the construction of \cite{GM20}, the $U$-walk obtained has the property that it crosses two sides of the original triangles only once. However, it is unlikely for a typical $U$-walk to satisfy this property, and one might expect to control the probability of this event using the natural renewal structure of bridges. Adapting the aforementioned construction we show that this is indeed the case in an averaged sense which allows us to obtain a more refined recurrence inequality -- see Corollary~\ref{cor: final ineq}. Our construction is more delicate though, as we end up with walks which might cross the real axis, and one needs to take into account the probability of this event. With this inequality in hand, we show that the value of $D_T$ drops by a constant factor along a geometric subsequence, from which the desired polynomial decay follows.

%\dimaText{At the end of the introduction mention that we believe that the $B_T$ bound should lead to better bounds, and then conjecture is that the displacement is at most $n^{1-\eps}$. The correct behaviour is ...}

\subsubsection*{Acknowledgements} We warmly thank Hugo Duminil-Copin for inspiring discussions. The project was initiated while the authors were at the University of Geneva. This research is supported by the Swiss National Science Foundation and the NCCR SwissMAP.

\section{Preliminary observations and notations}

Throughout this paper, we fix for convenience an embedding of the hexagonal lattice $\mathbb{H}$, where each edge is a line segment of length $\sqrt{3}/3$, and $0$ is the midpoint of a vertical edge, as shown in Figure~\ref{fig:domain}.

\subsection{Parafermionic observable}

Our analysis makes use of the parafermionic observable which was also used by Dumunil-Copin and Smirnov \cite{DCS12} to compute the connective constant of the hexagonal lattice. This observable is naturally defined in terms of self-avoiding walks between \emph{mid-edges} of $\mathbb{H}$, i.e.\ midpoints of edges of $\He$. We write $H$ for the set of mid-edges, and $E$ for the set of edges $\{a,b\}$ with $a,b\in H$ such that the corresponding edges $e_a$ and $e_b$ of $\mathbb{H}$ that contain $a$ and $b$, respectively, share a common vertex of $\mathbb{H}$. 

A \emph{walk} $\gamma$ on $H$ is a sequence $(\gamma_0,\gamma_1,\ldots,\gamma_n)$ such that for every $i=0,1,\ldots,n-1$, we have $\{\gamma_i,\gamma_{i+1}\}\in E$. A \emph{self-avoiding walk} (SAW) $\gamma$ is a walk such that $\gamma_i\neq \gamma_j$ for every $i\neq j$, i.e.\ every vertex is visited at most once. The \emph{length} $\ell(\gamma)$ of the walk is the number of vertices of $H$ visited by $\gamma$. We will write $\gamma:a\rightarrow S$ if a walk $\gamma$ starts at $a$ and ends at some mid-edge of $S\subset H$. In the case $S=\{b\}$, we simply write $\gamma:a\rightarrow b$. 

To define the observable, we need to restrict ourselves to walks living in domains. Given a set of vertices $V(\Omega)$ in $\mathbb{H}$, we let $E(\Omega)$ be the set of edges of $\mathbb{H}$ with at least one endpoint in $V(\Omega)$, and we let $\Omega\subset H$ be the set of mid-edges that lie in an edge of $E(\Omega)$. We call $\Omega$ a \emph{domain}. We also let $\partial \Omega$ be the set of mid-edges in $\Omega$ such that the corresponding edge in $E(\Omega)$ has one endpoint in the complement of $V(\Omega)$. See Figure~\ref{fig:domain}. In what follows, we always assume that $\Omega$ is simply connected, i.e.\ both $\Omega$ and $H\setminus \Omega$ span a connected graph. 

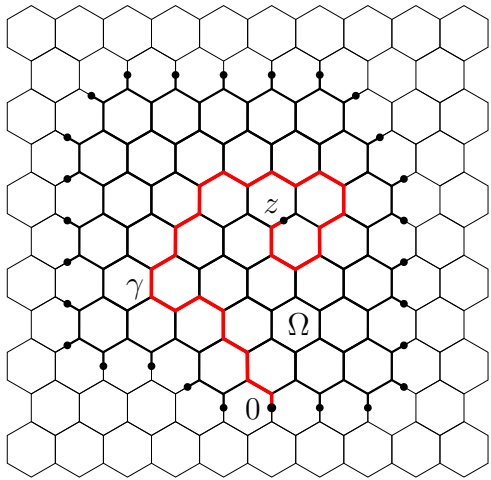
\begin{figure}
    \centering
    \includegraphics[width=.4\linewidth]{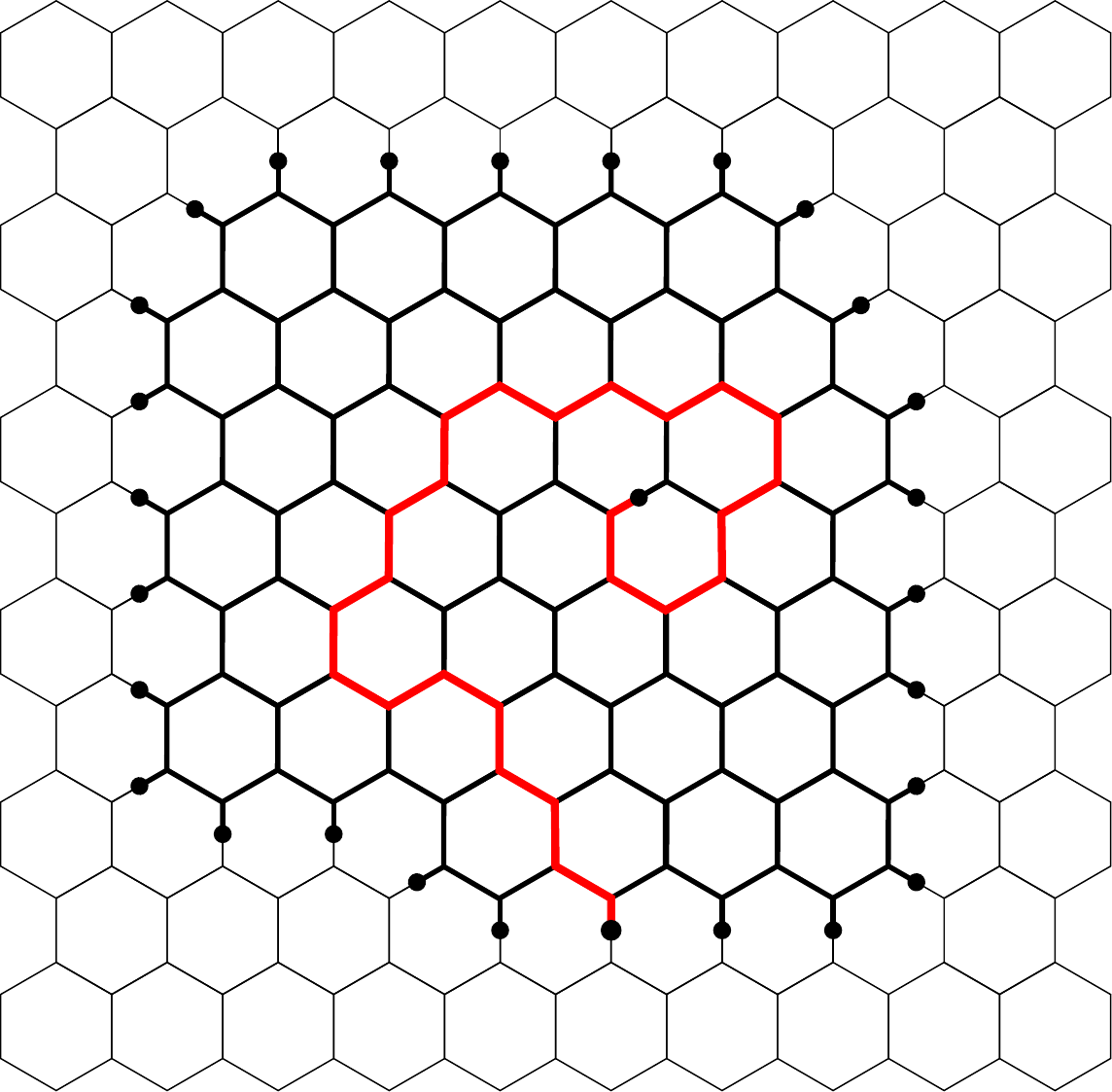}
    \put(-92,22){$0$}
    \put(-137,70){$\gamma$}
    \put(-85,100){$z$}
    \put(-76,54){$\Omega$}
    \caption{A domain $\Omega$ together with a self-avoiding walk $\gamma$ starting at $0$ and ending at $z$ shown in bold. The boundary $\partial \Omega$ of $\Omega$ and the start and end of $\gamma$ are depicted in dots.}
    \label{fig:domain}
\end{figure}

\begin{rem}
Consider the triangular lattice $\mathbb{T}$ embedded in the plane in such a way that vertices of $\mathbb{T}$ are the centers of the faces of $\mathbb{H}$ and the edges of $\mathbb{T}$ are line segments which meet the edges of $\mathbb{H}$ at the elements of $H$. Given a bounded and simply connected domain $\Omega\subset H$, the edges of $\mathbb{T}$ that intersect $\partial \Omega$ define a contour $\Pi$. 
\end{rem}

For a self-avoiding walk $\gamma$ starting at a mid-edge $a$ and ending at a mid-edge $b$, we define its
\emph{winding $\W_\gamma(a,b)$} as the total rotation of the direction in radians when $\gamma$ is traversed from $a$ to $b$. 

The parafermionic observable is defined as follows.

\begin{definition}
The \emph{parafermionic observable} for $a\in \partial \Omega$, $z\in \Omega$, 
is defined by
\begin{equation*}
	F(z)=F(a,z,x,\sigma)=\sum_{\gamma\subset \Omega:\ a\rightarrow z}\e^{-\I\sigma \W_\gamma(a,z)} x^{\ell(\gamma)}.
\end{equation*}
\end{definition}

\begin{rem}
The reason why we restrict to simply connected domains and why we choose $a\in \partial \Omega$ is because in this case for every $z\in \partial \Omega$, $F(z)$ is, up to a multiplicative constant, equal to the partition function of walks ending at $z$.
\end{rem}

To simplify formulas, below we set  $x_c:=1/\sqrt{2+\sqrt 2}$. The following result was obtained in \cite{DCS12}.

\begin{lemma}[\cite{DCS12}]\label{lem:relation}
If $x=x_c$ and $\sigma=\frac 58$, then $F$ satisfies the following relation for every vertex $v\in V(\Omega)$:
\[
		(p-v)F(p)+(q-v)F(q)+(r-v)F(r)=0,
\]
where $p,q,r$ are the mid-edges of the three edges adjacent to $v$.
\end{lemma}

\begin{rem}\label{rem: contour integral}
Given a domain $\Omega$, we can define a piecewise constant function, which by abuse of notation we also denote $F$, along the edges of $\mathbb{T}$ (here we view each edge as a copy of an open interval) as follows. For each edge $e$ of $\mathbb{T}$ that meets a mid-edge $p \in \Omega$, we let $G$ be equal to $F(p)$ along $e$.  Note that the dual graph of the three edges of $\mathbb{H}$ incident to $v\in V(\Omega)$ is a triangle, hence we can view the sum $(p-v)F(p)+(q-v)F(q)+(r-v)F(r)$ as the contour integral of $F$ along this triangle. In the case where $\Omega$ is enclosed by a contour $\Pi$, summing over the vertices of $V(\Omega)$ we have cancellations along the edges of $\mathbb{T}$ in the region bounded by $\Pi$, and we obtain that the corresponding contour integral of $F$ along $\Pi$ is equal to $0$.
\end{rem}

\subsection{Domains and partition functions of interest}

There are several types of domains important for our purposes, which we now introduce, together with the relevant partition functions.

The first case of interest is strips which give rise to bridges.
\begin{definition}\label{def: bridges}
Let $k\geq 0$ be an integer. We define $\textup{Strip}_k$ to be the strip of height $\frac{\sqrt{3}}{2}k$ defined by the lines $y=0$ and $y=\frac{\sqrt{3}}{2}k$. We define $B(\textup{Strip}_k)$ to be the set of \textit{bridges} of height $\frac{\sqrt{3}}{2}k$ which start at $0$, i.e.\ SAWs $\gamma:0\rightarrow \{y=\frac{\sqrt{3}}{2}k\}\cap H$ which are contained in $\textup{Strip}_k$.
We denote by $B_k$ the corresponding partition function at the critical point $x_c$, i.e.\
\[
B_k:=\sum_{\gamma\in B(\textup{Strip}_k)}x_c^{\ell(\gamma)}.
\] 
Occasionally, we will also refer to rotations and translations of walks in $B(\textup{Strip}_k)$ as bridges.
\end{definition}

The second case of interest is equilateral triangles which give rise to triangle walks. 

\begin{definition}\label{def: triangle}
Let $k\geq 0$ be an integer. We define $\textup{Tria}_{2k+1}$ to be the equilateral triangle of side-length $2k+1$ with vertices at the points $(-k-1/2,0)$, $(k+1/2,0)$ and $(0,\frac{\sqrt{3}}{2}(2k+1))$. We denote $\textup{R}_{2k+1}$ and $\textup{L}_{2k+1}$ the intersection of $H$  with the right and left side of $\textup{Tria}_{2k+1}$, respectively. We define $D(\textup{Tria}_{2k+1})$ to be the set of \textit{triangle walks} of width $\frac{\sqrt{3}}{2}(2k+1)$, i.e.\ SAWs $\gamma:0\rightarrow \textup{R}_{2k+1}\cup\textup{L}_{2k+1}$ which are contained in $\textup{Tria}_{2k+1}$. We let 
\[
D_{2k+1}:=\sum_{\gamma\in D(\textup{Tria}_{2k+1})}x_c^{\ell(\gamma)}
\] 
be the corresponding partition function. For convenience, we extend the definition to even integers by letting $D_{2k+2}:=D_{2k+1}$ and $D_0:=B_0/\cos\left(\frac{\pi}{8}\right)=1/\cos\left(\frac{\pi}{8}\right)$.
\end{definition}

Finally, we consider the upper half-plane.

\begin{definition}\label{def: U walk}
Let $\textup{U}:=\mathbb{R}\times [0,\infty)$ and denote by $G(\textup{U})$ the set of \textit{$U$-walks}, i.e.\ SAWs $\gamma:0\rightarrow (\mathbb{R}\times \{0\})\cap H$ which are contained in the upper half-plane $\textup{U}$. We define 
\[
G_k:=\sum_{\substack{\gamma \in G(\textup{U}): \\ \gamma:0\rightarrow (k,0)}} x_c^{\ell(\gamma)}
\]
to be the partition function of $U$-walks which end at $(k,0)$.
\end{definition}

% We write $B_T$ for the partition function of bridges on a strip of width $T$. 

% We write $\textup{Tria}_t$ for the equilateral triangle with vertices at (roughly) $(-T, 0), (T, 0)$ and $(0, \sqrt{3}T)$, so it has side of length $2T$. We write $D_T$ for the partition function of walks which start at $(0, 0)$ stay inside $\textup{Tria}_T$ and end at either the left of the right side of $\textup{Tria}$.

% We write $G_k$ for the partition function of walks from $(0, 0)$ to $(k, 0)$ which stay in the upper-half plane. 

Using the parafermionic observable we  obtain relations between the values of the three types of partition functions defined above. Some of these results were also obtained in \cite{GM20} and our proofs are based on similar arguments.

\begin{lemma}\label{lem: G sum}
We have that
\[
\sum_{k=1}^\infty G_k \leq \frac{1}{2\cos\left(\frac{3\pi}{8}\right)}.
\]
Furthermore, for every $T\geq 1$ we have
\[
\sum_{k=T}^\infty G_k \leq \frac{\cos\left(\frac{\pi}{8}\right)}{2\cos\left(\frac{3\pi}{8}\right)}D_{2T-1}. 
\]
\end{lemma}
\begin{proof}
For the first part, let $A_T$ be the partition function of SAWs $\gamma:0\rightarrow (\mathbb{R}\times \{0\})\cap H$ which live in $\textup{Strip}_T$ and have $\ell(\gamma)>0$.
Using the parafermionic observable one can show, see \cite[Corollary 2.2]{GM20} for the detailed proof, that for every $T\geq 1$ we have
\begin{equation}\label{eq: B_T local}
\cos\left(\frac{3 \pi}{8}\right) A_T+B_T=1.
\end{equation}
Since $A_T$ converges to $\sum_{k\in \mathbb{Z}\setminus \{0\}} G_k=2\sum_{k=1}^{\infty} G_k$ as $T$ tends to infinity and $B_T\geq 0$, the desired result follows.

For the second part, we again use the parafermionic observable. Let $A^{\triangle}_{2T-1}$ be the partition of SAWs $\gamma:0\rightarrow (\mathbb{R}\times \{0\})\cap H$ which live in $\textup{Tria}_{2T-1}$ and have $\ell(\gamma)>0$. Recall Lemma~\ref{lem:relation} and Remark~\ref{rem: contour integral}. The contour integral of $F$ along the boundary of $\textup{Tria}_{2T-1}$ being $0$ gives 
\begin{equation}\label{eq: D_T local}
\cos\left(\frac{3 \pi}{8}\right) A^{\triangle}_{2T-1}+\cos\left(\frac{\pi}{8}\right) D_{2T-1}=1,
\end{equation}
where the term on the right-hand side comes from the trivial walk of length $0$. Using the first part of the lemma we obtain
\[
2\cos\left(\frac{3 \pi}{8}\right)\sum_{k=1}^{\infty} G_k-\cos\left(\frac{3 \pi}{8}\right)A^{\triangle}_{2T-1}\leq 1-\cos\left(\frac{3 \pi}{8}\right)A^{\triangle}_{2T-1}=\cos\left(\frac{\pi}{8}\right) D_{2T-1}.
\]
Since all walks contributing to $A^{\triangle}_{2T-1}$ contribute also to $2\sum_{k=1}^{T-1} G_k$, we have that 
\[
2\sum_{k=1}^{\infty} G_k-A^{\triangle}_{2T-1}\geq 2\sum_{k=T}^{\infty} G_k,
\]
from which the desired result follows.
\end{proof}

\begin{rem}
In fact, since $B_T$ tends to zero as $T$ tends to infinity \cite{GM20}, the first inequality of Lemma \ref{lem: G sum} is an equality. However, we will not need this fact. 
\end{rem}

\begin{lemma}\label{lem: B-D ineq}
The sequences $(B_k)_{k\geq 0}$ and $(D_k)_{k\geq 0}$ are non-increasing and  
\[
B_k\leq \cos\left(\frac{\pi}{8}\right) D_k
\]
for every $k\geq 0$.
\end{lemma}
\begin{proof}
We start with $(B_k)_{k\geq 0}$. Applying \eqref{eq: B_T local} for $T=k\geq 0$ and $T=k+1$ we obtain   
\[
B_k-B_{k+1}=\cos\left(\frac{3\pi}{8}\right) \left(A_{k+1}-A_k\right).
\]
Since all walks contributing to $A_k$ contribute also to $A_{k+1}$, we have that $A_{k+1}\geq A_k$, which implies that $B_{k+1}\leq B_k$, as desired.

For the monotonicity of $(D_k)_{k\geq 0}$, first note that $D_1\leq D_0:=1/\cos\left(\frac{\pi}{8}\right)$ follows easily from \eqref{eq: D_T local}. Since $D_{2k+2}:=D_{2k+1}$ by definition, the monotonicity of $(D_k)_{k\geq 0}$ is then equivalent to the assertion that $D_{2k-1}\leq D_{2k+1}$ for all $k\geq 1$. Since $A^{\triangle}_{2T-1}$ is trivilly non-decreasing, this follows from \eqref{eq: D_T local}.

Finally, we compare the two sequences. Note that for every $k\geq 0$,
\[
B_{2k+1}-\cos\left(\frac{\pi}{8}\right)D_{2k+1}= A^{\triangle}_{2k+1}-A_{2k+1}\leq 0,
\]
which proves the desired inequality for odd indices. For the case of even indices, we have $B_0= \cos\left(\frac{\pi}{8}\right) D_0$ by definition, and furthermore, for every $k\geq 1$ we have
\[
B_{2k}\leq B_{2k-1}\leq \cos\left(\frac{\pi}{8}\right) D_{2k-1}=\cos\left(\frac{\pi}{8}\right) D_{2k}.
\]
\end{proof}

%\section{Outline of the proof}

\section{Polynomial decay of $B_T$}

In this section, we prove that $D_T$, and hence $B_T$, decays polynomially fast to $0$.

Our starting point is a construction by Glazman and Manolescu \cite{GM20} where they concatenated three triangle walks to form a $U$-walk. This construction led them to the inequality 
\begin{equation}\label{eq: Shasha Mano}
(D_{8T+1})^3\leq 8\sum_{k=T}^{9T} G_k.    
\end{equation}
See Figure~\ref{fig: 3 triangles}. Our aim is to improve on the latter inequality by showing that $D_{8T+1}$ is comparable with $\sum_{k=T}^{cT} G_k$ for some $c>0$,
which combined with Lemma~\ref{lem: G sum} implies that the infinite sum $\sum_{k=T+1}^{\infty} G_k$ is comparable with the finite sum $\sum_{k=T}^{cT} G_k$.
This comparison holds only if $\sum_{k=T+1}^{\infty} G_k$ decays polynomially fast to $0$, which in turn implies a similar decay for $D_{T}$. The actual inequality we obtain is more involved, see Corollary \ref{cor: final ineq}.

\begin{figure}[ht]
    \centering
    \includegraphics[width=.5\linewidth]{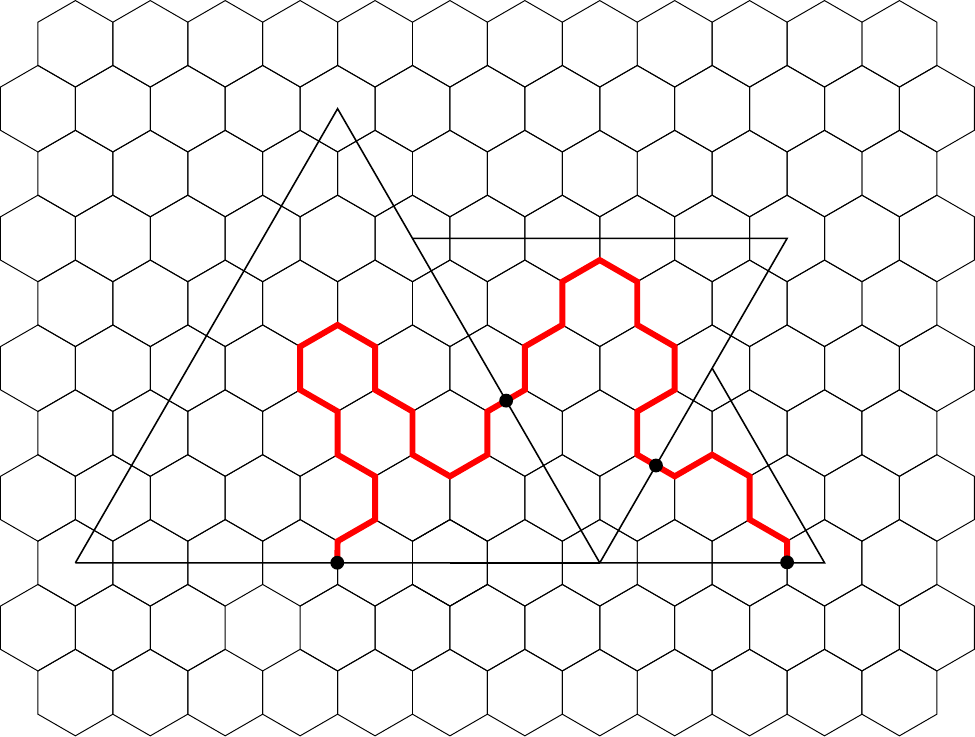}
     \caption{The concatenation of three triangle walks that creates a $U$-walk.}
    \label{fig: 3 triangles}
\end{figure}

\subsection{Number of renewal times}

We start by introducing the following notion.
Let $k\geq 0$ be an integer. We say that an integer $0\leq i\leq k$ is a renewal time for a SAW $\gamma\in D(\textup{Tria}_{2k+1})$ if $\gamma$ intersects the line $i+1/2+e^{2\pi i/3} \mathbb{R}$ exactly once\footnote{For our purposes, we are interested only in walks ending on $\textup{R}_{2k+1}$. For this reason, we define renewal times only with respect to lines parallel to $\textup{R}_{2k+1}$.}. Note that by definition, $k$ is a renewal time whenever the walk ends on the right side of $\textup{Tria}_{2k+1}$. We write $N=N(\gamma)$ for the number of renewal times of $\gamma$.

In what follows, $\mathbb{P}^{\textup{Tria}}_{2k+1}$ denotes the probability measure on $D(\textup{Tria}_{2k+1})$ defined by 
\[
\mathbb{P}^{\textup{Tria}}_{2k+1}[\gamma]=\frac{x_c^{\ell(\gamma)}}{D_{2k+1}}.
\]
We write $\mathbb{E}^{\textup{Tria}}_{2k+1}$ for the expectation under $\mathbb{P}^{\textup{Tria}}_{2k+1}$.

\begin{lemma}\label{lem: expecation}
For every $k\geq 1$ we have
\[
\mathbb{E}^{\textup{Tria}}_{2k+1}[N]\leq 2\cos\left(\frac{\pi}{8}\right)\frac{D_{\lceil \frac{k+1}{2} \rceil}}{D_{2k+1}}\sum_{i=0}^{\lfloor \frac{k+1}{2} \rfloor} D_i. 
\]
\end{lemma}
\begin{proof}
Let $i\geq 0$. We estimate the probability that $i$ is a renewal time. Consider a SAW $\gamma\in D(\textup{Tria}_{2k+1})$ such that $i$ is a renewal time. Let $x$ be the mid-edge in the line $i+1/2+e^{2\pi i/3}\mathbb{R}$ that $\gamma$ visits. Note that the sub-walk of $\gamma$ from $x$ until its endpoint is (the rotation and translation of) a bridge within $\textup{Strip}_{k-i}$; whilst the sub-walk of $\gamma$ from $0$ to $x$ lies in the triangle $T_{k,i}$ which is defined as the intersection of $\textup{Tria}_{2k+1}$ with the half-plane on the left side of the line $i+1/2+e^{2\pi i/3}\mathbb{R}$. See Figure~\ref{fig: renewal}. Let us write $R_{k,i}$ and $L_{k,i}$ for the right and left side of $T_{k,i}$, respectively, and define
\[
\Delta_{k,i}=\sum_{\substack{\gamma\subset T_{k,i}:\\ \gamma:0\rightarrow R_{k,i}\cup L_{k,i}}} x_c^{\ell(\gamma)}.
\]

\begin{figure}[ht]
    \centering
    \includegraphics[width=.4\linewidth]{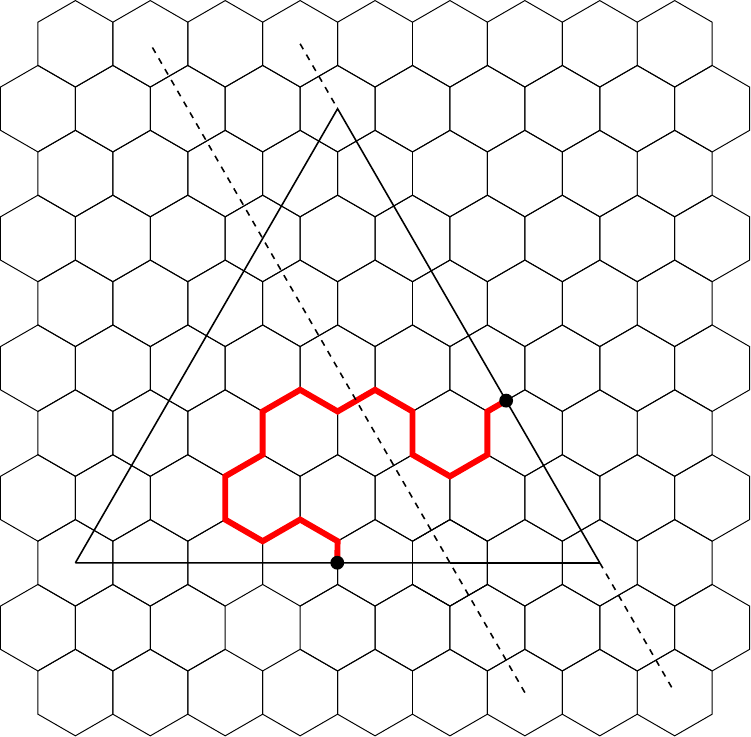}
    \caption{The decomposition of a walk in the proof of Lemma~\ref{lem: expecation}. The dashed line on the left gives rise to an (off-centered) equilateral triangle and together with the other dashed line it defines a strip.}
    \label{fig: renewal}
\end{figure}

Since there are more SAWs in $T_{k,i}$ that start at $0$ and end on the real axis than SAWs in $\textup{Tria}_{2i+1}$ that start at $0$ and end on the real axis, we can use the observable as in the proof of Lemma~\ref{lem: B-D ineq} to deduce that $\Delta_{k,i}\leq D_{2i+1}$.
Thus 
\[
\mathbb{P}^{\textup{Tria}}_{2k+1}[i \text{ is a renewal time}]\leq \frac{D_{2i+1}B_{k-i}}{D_{2k+1}}.
\]

Summing over all $i$ we obtain
\begin{equation*}
\mathbb{E}^{\textup{Tria}}_{2k+1}[N]
\leq 
\frac{1}{D_{2k+1}}\sum_{i=0}^k D_{2i+1}B_{k-i}
\leq  
\frac{\cos\left(\frac{\pi}{8}\right)}{D_{2k+1}}\sum_{i=0}^k D_{i+1}D_{k-i}
\leq 2\cos\left(\frac{\pi}{8}\right) \frac{D_{\lceil \frac{k+1}{2} \rceil}}{D_{2k+1}}\sum_{i=0}^{\lfloor \frac{k+1}{2} \rfloor} D_i.
\end{equation*}
In the second inequality we used that $D_{2i+1}\leq D_{i+1}$ and $B_{k-i}\leq \cos\left(\frac{\pi}{8}\right) D_{k-i}$. In the third inequality we used that $D_{i+1}\leq D_{\lceil \frac{k+1}{2} \rceil}$ for $i\geq \lceil \frac{k+1}{2} \rceil-1$, and similarly, $D_{k-i}\leq D_{\lceil \frac{k+1}{2} \rceil}$ for $i< \lceil \frac{k+1}{2} \rceil-1$. This completes the proof.
\end{proof}

\subsection{Geometric construction and recurrence inequalities}

With the above result in hands, we now start describing our construction. Along the way we need to consider two cases.

\textbf{Construction:} Let $T\geq 1$. For each $T\leq k \leq 2T-1$, we consider a walk $\gamma_1\in D(\textup{Tria}_{2k+1})$ which ends on the right side of $\textup{Tria}_{2k+1}$. Now depending on the endpoint $x$ of $\gamma_1$, for every $4T\leq i \leq 5T-1$, we let $\textup{Tria}_{2i+1,x}:=x+e^{-\pi i/3}\textup{Tria}_{2i+1}$. Note that for every choice of $i$, $\textup{Tria}_{2i+1,x}$ intersects the real axis $\mathbb{R}\times \{0\}$ and it has non-trivial intersection with the lower half-plane. We consider a SAW $\gamma_2$ lying in $\textup{Tria}_{2i+1,x}$ which starts at $x$ and ends on the right side of $\textup{Tria}_{2i+1,x}$. Now either $\gamma_2$ does not intersect $\mathbb{R}\times \{0\}$ or it does. In the first case, we glue another triangle SAW $\gamma_3$ in order to obtain a $U$-walk. In the second case, we use the observable to “cut” the part of $\gamma_2$ after its first intersection with the real axis. In order to handle the number of ways to reconstruct each $U$-walk we obtained, we need to restrict ourselves to SAWs $\gamma_1,\gamma_2$ whose number of renewal times is not much larger than its mean. 

We now describe the above construction in more details. We define $D^-_{2i+1,x}$ to be the partition function of SAWs contained in $\textup{Tria}_{2i+1,x}$ which start at $x$, end on the right side of $\textup{Tria}_{2i+1,x}$ and intersect the real axis.

For each mid-edge $x$ on the right side of $\textup{Tria}_{2k+1}$ there are two possibilities: Either there exists $4T\leq i_x\leq 5T-1$ such that 
\begin{equation}\label{eq: def i_x}
D^-_{2i_x+1,x} \geq D_{2i_x+1}/4
\end{equation} 
or 
\begin{equation}\label{eq: non i_x}
D^-_{2i+1,x}<D_{2i+1}/4 \text{  for every  } 4T\leq i \leq 5T-1.
\end{equation}
Let us write $\Sigma_{2k+1}$ for the set of $x\in \textup{R}_{2k+1}$ for which the former happens. We now consider the following two cases:
\begin{enumerate}[label=\textbf{(\alph*)}]
   \item \label{case: below line} $\sum_{k=T}^{2T-1}\sum_{x\in \Sigma_{2k+1}} D_{2k+1}(x) \geq \sum_{k=T}^{2T-1} D_{2k+1}/4$,
    \item \label{case: above line} $\sum_{k=T}^{2T-1}\sum_{x\in \Sigma_{2k+1}} D_{2k+1}(x) < \sum_{k=T}^{2T-1} D_{2k+1}/4$.
\end{enumerate}
\noindent
We will handle each of these cases separately. 

Let us start from the first case. We will need the following definitions. For each $x\in \Sigma_{2k+1}$, let us fix an integer $4T\leq i_x\leq 5T-1$ such that $D^-_{2i_x+1,x} \geq D_{2i_x+1}/4$. Write $\textup{Trap}_{2i+1,x}$ for the trapezoid which is defined as the intersection of $\textup{Tria}_{2i+1,x}$ with the upper half-plane $\textup{U}$. Let $F^R_{2i+1,x}, F^L_{2i+1,x}, F^T_{2i+1,x}, F^B_{2i+1,x}$ be the partition functions of SAWs contained in $\textup{Trap}_{2i+1,x}$ which start at $x$ and end on the right, left, top and bottom side of $\textup{Trap}_{2i+1,x}$, respectively.

For a mid-edge $x$ on the right or the left side of $\textup{Tria}_{2k+1}$ we define 
\[
D_{2k+1}(x):=\sum_{\substack{\gamma\in D(\textup{Tria}_{2k+1}):\\ \gamma:0\rightarrow x}} x_c^{\ell(\gamma)}
\] 
and
\[
D^\textup{ren}_{2k+1}(x):=\sum_{\substack{\gamma\in D(\textup{Tria}_{2k+1}):\\ \gamma:0\rightarrow x,\\ N(\gamma)\leq M_k}} x_c^{\ell(\gamma)},
\] 
where 
\[
M_k:=16\cos\left(\frac{\pi}{8}\right)\frac{D_{\lceil \frac{k+1}{2} \rceil}}{D_{2k+1}}\sum_{i=0}^{\lfloor \frac{k+1}{2} \rfloor} D_i,
\]
which is to be compared with the bound in Lemma~\ref{lem: expecation}.

\begin{lemma}\label{lem: first case}
Let $T\geq 1$. Assume that case~\ref{case: below line} happens. Then 
\begin{equation}\label{eq:inequality-case-a}
T D^2_{10T}\leq 256\sqrt{2}\frac{D_{\lceil \frac{T+1}{2} \rceil}}{D_{4T}}\sum_{i=0}^T D_i \sum_{\ell=T}^{7T} G_\ell.
\end{equation}
\end{lemma} 

\begin{figure}
    \centering
    \includegraphics[width=.5\linewidth]{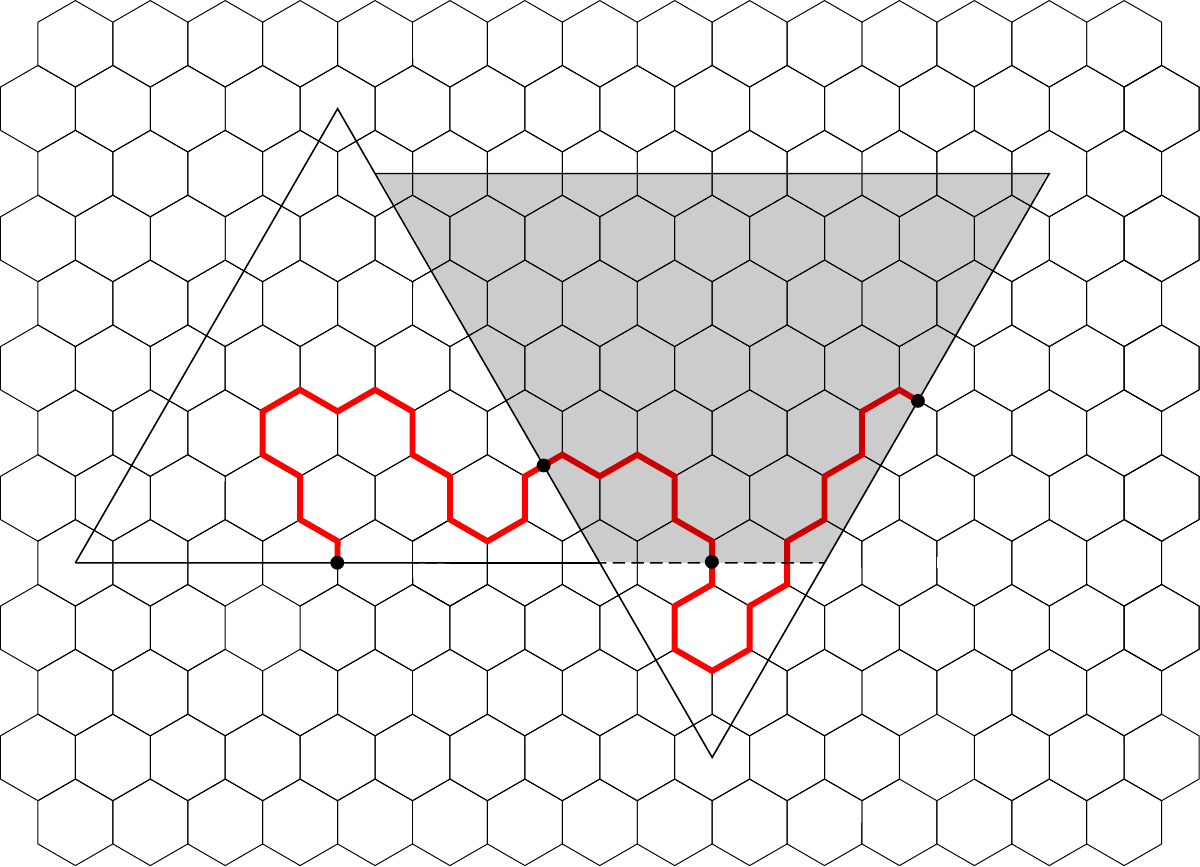}
     \caption{The construction in the case~\ref{case: below line}. The part of the real line intersecting the triangle on the right is indicated in dashed lines. From the walk in the triangle on the right we only keep its part before it hits the dashed line. The trapezoid $\textup{Trap}_{2i+1, x}$ is highlighted in gray.}
    \label{fig: below line}
\end{figure}

\begin{proof}
Our strategy will be to concatenate walks contributing to $\sum_{k=T}^{2T}\sum_{x\in \Sigma_{2k+1}} D_{2k+1}(x)$ with walks contributing to $F^B_{2i_x+1,x}$ to obtain a $U$-walk. See Figure~\ref{fig: below line}. Since each of the $U$-walks might have multiple pre-images, we work with $D^\textup{ren}_{2k+1}(x)$ instead of $D_{2k+1}(x)$ to control the number of the latter pre-images.

We start by noting that 
\[
\sum_{k=T}^{2T-1}\sum_{x\in \Sigma_{2k+1}} D^\textup{ren}_{2k+1}(x)\geq \sum_{k=T}^{2T-1}\sum_{x\in \Sigma_{2k+1}} D_{2k+1}(x)-\sum_{k=T}^{2T-1} D_{2k+1}/8
\]
by Markov's inequality and Lemma \ref{lem: expecation}. Then using our assumption we obtain 
\begin{equation}\label{eq: D star}
\sum_{k=T}^{2T-1}\sum_{x\in \Sigma_{2k+1}} D^\textup{ren}_{2k+1}(x)\geq \sum_{k=T}^{2T-1} D_{2k+1}/8.
\end{equation}

We now estimate $F^B_{2i_x+1,x}$. Recall Lemma~\ref{lem:relation} and Remark~\ref{rem: contour integral}. The contour integral along the boundary of the domains $\textup{Trap}_{2i_x+1,x}$ and $\textup{Tria}_{2i_x+1,x}$ being $0$ gives
\[
\cos\left(\frac{3\pi}{8}\right) F^L_{2i_x+1,x}+\cos\left(\frac{\pi}{8}\right) F^R_{2i_x+1,x}+ \cos\left(\frac{\pi}{8}\right) F^T_{2i_x+1,x}+\cos\left(\frac{\pi}{4}\right)F^B_{2i_x+1,x}=1
\]
and 
\[
\cos\left(\frac{3\pi}{8}\right) A_{2i_x+1}+\cos\left(\frac{\pi}{8}\right) D_{2i_x+1}=1.
\]
Thus
\[
F^B_{2i_x+1,x}= \frac{1}{\cos\left(\frac{\pi}{4}\right)}\left(\cos\left(\frac{3\pi}{8}\right) (A_{2i_x+1}-F^L_{2i_x+1,x})+\cos\left(\frac{\pi}{8}\right)(D_{2i_x+1}-F^R_{2i_x+1,x}-F^T_{2i_x+1,x})\right).
\]
Note that $A_{2i_x+1}-F^L_{2i_x+1,x}\geq 0$ since any walk contained in $\textup{Trap}_{2i_x+1,x}$ is also contained in $\textup{Tria}_{2i_x+1,x}$;
and $D_{2i_x+1}-F^R_{2i_x+1,x}-F^T_{2i_x+1,x}\geq D^-_{2i_x+1,x}$, since any walk contributing to $D^-_{2i_x+1,x}$ is contained in $\textup{Tria}_{2i_x+1,x}$ but not in $\textup{Trap}_{2i_x+1,x}$. This implies that
\[
F^B_{2i_x+1,x}\geq  \frac{\cos\left(\frac{\pi}{8}\right)}{\cos\left(\frac{\pi}{4}\right)}D^-_{2i_x+1,x}\geq \frac{\cos\left(\frac{\pi}{8}\right)}{4\cos\left(\frac{\pi}{4}\right)}D_{2i_x+1,x},
\]
where in the last inequality we used \eqref{eq: def i_x}.
Combining this with \eqref{eq: D star} and the monotonicity of $D_i$ we obtain 
\begin{equation}\label{eq: long ineq}
\begin{aligned}
\sum_{k=T}^{2T-1}\sum_{x\in \Sigma_{2k+1}} D^\textup{ren}_{2k+1}(x) F^B_{2i_x+1,x} \geq & \frac{\cos\left(\frac{\pi}{8}\right)}{4\cos\left(\frac{\pi}{4}\right)}\sum_{k=T}^{2T-1}\sum_{x\in \Sigma_{2k+1}} D^\textup{ren}_{2k+1}(x)D_{2i_x+1,x} \\ \geq &\frac{\cos\left(\frac{\pi}{8}\right)}{4\cos\left(\frac{\pi}{4}\right)} D_{10T} \sum_{k=T}^{2T-1}\sum_{x\in \Sigma_{2k+1}} D^\textup{ren}_{2k+1}(x) \\ \geq & \frac{\cos\left(\frac{\pi}{8}\right)}{32\cos\left(\frac{\pi}{4}\right)} D_{10T}\sum_{k=T}^{2T-1} D_{2k+1}\\
\geq & \frac{\cos\left(\frac{\pi}{8}\right)}{32\cos\left(\frac{\pi}{4}\right)} T D^2_{10T}.
\end{aligned}
\end{equation}

To upper bound the left hand side of the last inequality, note that concatenating a walk contributing to $D^\textup{ren}_{2k+1}(x)$ with a walk contributing to $F^B_{2i_x+1,x}$ results in a $U$-walk that contributes to $\sum_{\ell=T}^{7T} G_\ell$. It remains to estimate the number of ways to construct each $U$-walk in this sum. To this end, consider such a walk $\gamma$ and let $(\gamma_1,\gamma_2)$ be a pre-image of $\gamma$ such that $\gamma_1$ contributes to $D^\textup{ren}_{2k+1}(x)$ and $\gamma_2$ contributes to $F^B_{2i_x+1,x}$ for a certain $x$, with $k\in [T, 2T-1]$ being maximal among the pre-images. Then any other pre-image corresponds to a renewal time of $\gamma_1$, and by definition, the number of renewal times of $\gamma_1$ is at most $M_k$. Thus the number of pre-images is at most $M_k$, and by the multivalued map principle we have
\begin{equation}\label{eq:multivalued-bound-case-a}
 \sum_{k=T}^{2T-1}\sum_{x\in \Sigma_{2k+1}} D^\textup{ren}_{2k+1}(x) F^B_{2i_x+1,x} \leq \left(\max_{k \in [T, 2T-1]} M_k\right)\cdot \sum_{\ell=T}^{7T} G_\ell.
\end{equation}
 By the monotonicity of $D_k$ we have that for every $k\in [T, 2T-1]$ one has
\[
M_k\leq 16 \cos\left(\frac{\pi}{8}\right)\frac{D_{\lceil \frac{T+1}{2} \rceil}}{D_{4T-1}}\sum_{i=0}^T D_i.
\]
Thus, putting this together with inequalities \eqref{eq:multivalued-bound-case-a} and \eqref{eq: long ineq}  we obtain
\[
T D^2_{10T}\leq 512\cos\left(\frac{\pi}{4}\right)\frac{D_{\lceil \frac{T+1}{2} \rceil}}{D_{4T-1}}\sum_{i=0}^T D_i \sum_{\ell=T}^{7T} G_\ell,
\]
and the desired assertion follows.
\end{proof}

We now handle the second case. In what follows, we say that $k$ is a renewal time of some SAW $\gamma$ contributing to $F^R_{2i+1,x}$ if it is a renewal time of $e^{\pi i/3}(\gamma-x)$. We write $F^{R,\textup{ren}}_{2i+1,x}$ for the partition function of walks contributing to $F^R_{2i+1,x}$ with at most $M_i$ renewal times.

\begin{lemma}\label{lem: second case}
Let $T\geq 1$. Assume that case \ref{case: above line} happens. Then 
\begin{equation}\label{eq:inequality-case-b}
T^2 D^3_{18T}\leq 2^{15}\left(\frac{D_{\lceil \frac{T+1}{2} \rceil}}{D_{10T}}\sum_{i=0}^{\lfloor \frac{5T}{2} \rfloor} D_i\right)^2 \sum_{k=T}^{21T} G_k.
\end{equation}
\end{lemma}
\begin{proof}
We begin with the following construction.
Let $\gamma_1$ be a SAW contributing to $\sum_{k=T}^{2T-1}\sum_{x\in R_{2k+1}\setminus \Sigma_{2k+1}} D_{2k+1}(x)$. Let also $\gamma_2$ be a SAW contributing to $F^{R,\textup{ren}}_{2i+1,x}$ for some $4T\leq i \leq 5T-1$, where $x$ is the endpoint of $\gamma_1$. We denote $y$ the endpoint of $\gamma_2$, and let $\gamma_3$ be a SAW in $y+e^{4\pi i/3}D(\textup{Tria}_{2r})$, where $r$, a half-integer, is the distance of $y$ from the lowest-most point of the right side of $\textup{Trap}_{2i+1,x}$. Concatenating $\gamma_1$, $\gamma_2$ and $\gamma_3$ we obtain a walk $\gamma$ contributing to $\sum_{k=T}^{21T} G_k$. For the upper limit of this sum, we used that the endpoint of $\gamma$ has distance at most $2T+\ell+2r$ from $0$, where $\ell$ is the length of the bottom side of $\textup{Trap}_{2i+1,x}$, and also that $\ell\geq i-4T$, and $r\leq 2i-\ell$. See Figure~\ref{fig: above line}. 

\begin{figure}
    \centering
    \includegraphics[width=.5\linewidth]{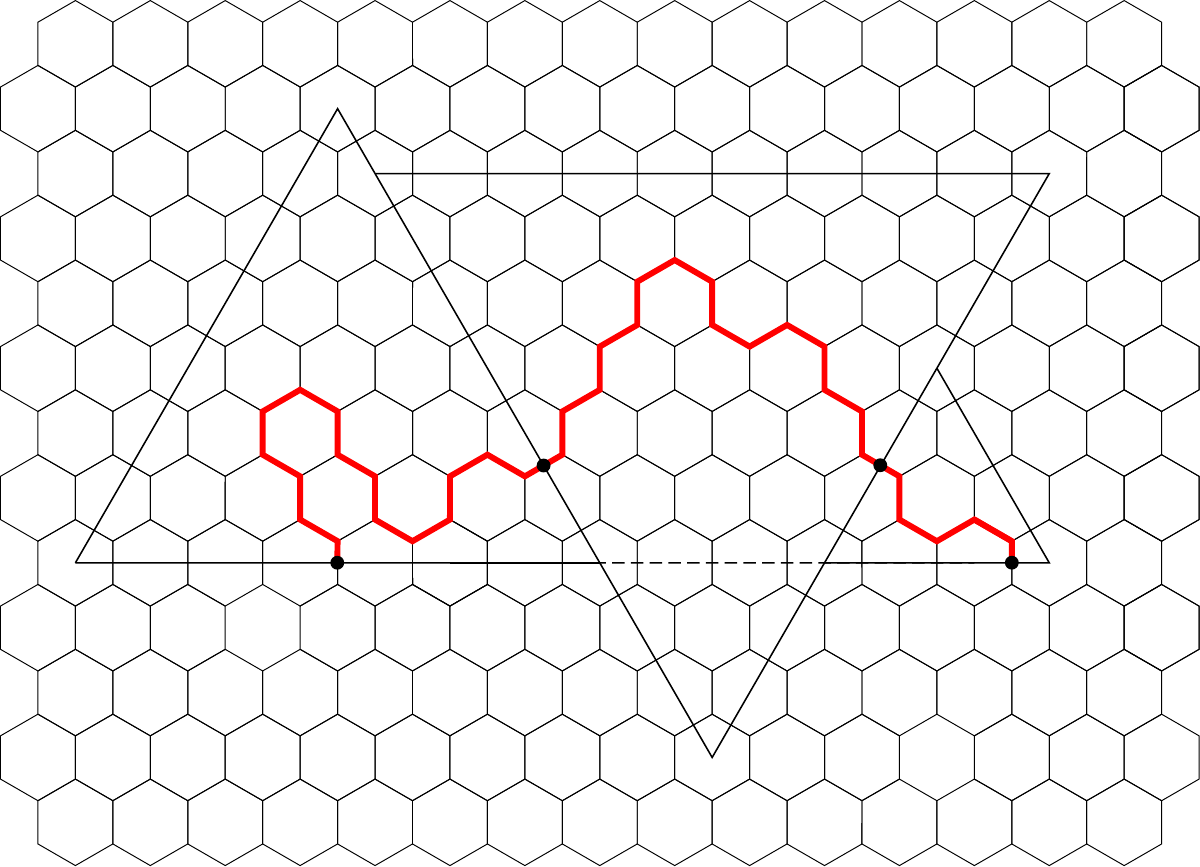}
     \caption{The construction in the case~\ref{case: above line}.}
    \label{fig: above line}
\end{figure}

%$2T+\ell+2r\leq 2T+\ell+2(2i-\ell)\leq 2T+4i-\ell \leq 2T+4i-(i-4T) = 6T+3i = 21T$

We now estimate the partition function of the constructed walks. Note that 
\[
\sum_{k=T}^{2T-1}\sum_{x\in \textup{R}_{2k+1}\setminus\Sigma_{2k+1}} D^\textup{ren}_{2k+1}(x)\geq \sum_{k=T}^{2T-1}\sum_{x\in \textup{R}_{2k+1}\setminus \Sigma_{2k+1}} D_{2k+1}(x)-\sum_{k=T}^{2T} D_{2k+1}/8
\]
by Markov's inequality and Lemma \ref{lem: expecation}. Hence using our assumption we obtain 
\[
\sum_{k=T}^{2T-1}\sum_{x\in \textup{R}_{2k+1}\setminus \Sigma_{2k+1}} D^\textup{ren}_{2k+1}(x)\geq \sum_{k=T}^{2T-1} D_{2k+1}/8.
\]  
Moreover, using Markov's inequality, Lemma \ref{lem: expecation} and \eqref{eq: non i_x} we obtain
\[
F^{R,\textup{ren}}_{2i+1,x}\geq F^R_{2i+1,x}-\frac{1}{8}D_{2i+1}=\frac{1}{2}D_{2i+1}-D^-_{2i+1,x}-\frac{1}{8}D_{2i+1}\geq \frac{1}{8}D_{2i+1}
\]
for every $i\in \textup{R}_{2k+1}\setminus \Sigma_{2k+1}$.
This implies that the partition function of the constructed walks satisfies
\[
\sum_{k=T}^{2T-1}\sum_{x\in \textup{R}_{2k+1}\setminus \Sigma_{2k+1}} D^\textup{ren}_{2k+1}(x)\sum_{i=4T}^{5T-1} F^{R,\textup{ren}}_{2i+1,x}\frac{D_{2r}}{2}\geq \frac{T^2}{128} D^3_{18T},
\]
where we used that $r<9T$.

To establish an upper bound for the left-hand side of the above inequality, we need to estimate the number of pre-images for each SAW $\gamma$ constructed. Arguing as in the proof of Lemma~\ref{lem: first case} we see that there are at most 
\[
\max_{k\in [T, 2T-1]} M_k\leq 16 \cos\left(\frac{\pi}{8}\right)\frac{D_{\lceil \frac{T+1}{2} \rceil}}{D_{4T-1}}\sum_{i=0}^T D_i \leq 16 \cos\left(\frac{\pi}{8}\right)\frac{D_{\lceil \frac{T+1}{2} \rceil}}{D_{10T}}\sum_{i=0}^{\lfloor \frac{5T}{2} \rfloor} D_i
\]
choices for $\gamma_1$, and at most
\[
\max_{i\in [4T, 5T-1]} M_i\leq 16 \cos\left(\frac{\pi}{8}\right)\frac{D_{\lceil \frac{4T+1}{2} \rceil}}{D_{10T-1}}\sum_{i=0}^{\lfloor \frac{5T}{2} \rfloor} D_i \leq 16 \cos\left(\frac{\pi}{8}\right)\frac{D_{\lceil \frac{T+1}{2} \rceil}}{D_{10T}}\sum_{i=0}^{\lfloor \frac{5T}{2} \rfloor} D_i
\]
choices for $\gamma_2$. Therefore, combining the above inequalities we obtain that
\[
T^2 D^3_{18T}\leq 2^{15}\left(\frac{D_{\lceil \frac{T+1}{2} \rceil}}{D_{10T}}\sum_{i=0}^{\lfloor \frac{5T}{2} \rfloor} D_i\right)^2 \sum_{k=T}^{21T} G_k,
\]
as desired. 
\end{proof}

We now obtain the following result which holds for any value of $T$, regardless of whether case \ref{case: below line} or \ref{case: above line} happens.

\begin{corollary}\label{cor: final ineq}
For every $T\geq 1$ one has
\begin{equation}\label{eq:unified-bound}
T^4D_{18T}^5\leq 2^{17} \left(\sum_{i=0}^{3T} D_i\right)^4\cdot \sum_{k=T}^{21T}G_k.    
\end{equation}
\end{corollary}
\begin{proof}
We first observe that the bound \eqref{eq:inequality-case-b} is satisfied for all $T\geq 1$. Indeed, let $T\geq 1$. If case \ref{case: above line} happens then this follows directly from Lemma \ref{lem: second case}. Otherwise, if case \ref{case: below line} happens, then we deduce \eqref{eq:inequality-case-b} from \eqref{eq:inequality-case-a} by observing that by the monotonicity of $(D_k)$
\[
\frac{D_{\lceil \frac{T+1}{2} \rceil}}{D_{10T}}\sum_{i=0}^{\lfloor \frac{5T}{2} \rfloor} D_i \geq \sum_{i=0}^{\lfloor \frac{5T}{2} \rfloor} D_i \geq T D_{18T}, 
\]
hence
\begin{equation*}
\begin{aligned}
T^2 D_{18T}^3 &\leq T^2 D_{10T}^2 D_{18T}\leq 256\sqrt{2}\cdot T D_{18T}\frac{D_{\lceil \frac{T+1}{2} \rceil}}{D_{4T}}\sum_{i=0}^T D_i \sum_{k=T}^{7T} G_k\\
&\leq 2^{15}T D_{18T}\frac{D_{\lceil \frac{T+1}{2} \rceil}}{D_{10T}}\sum_{i=0}^{\lfloor \frac{5T}{2} \rfloor} D_i \sum_{k=T}^{21T}G_k\\
&\leq 2^{15}\left(\frac{D_{\lceil \frac{T+1}{2} \rceil}}{D_{10T}}\sum_{i=0}^{\lfloor \frac{5T}{2} \rfloor} D_i \right)^2 \sum_{k=T}^{21T}G_k
\end{aligned}
\end{equation*}

To deduce \eqref{eq:unified-bound} from \eqref{eq:inequality-case-b}  note that by the monotonicity of $(D_k)$
\[
T^2D_{18T}^2\cdot \left(\frac{D_{\lceil \frac{T+1}{2} \rceil}}{D_{10T}}\right)^{2}\leq T^2D_{\lceil \frac{T+1}{2} \rceil}^2 \leq
\left(\frac{T}{\lceil \frac{T+1}{2} \rceil}\sum_{i=0}^{\lceil\frac{T+1}{2}\rceil} D_i\right)^2\leq 4 \left(\sum_{i=0}^{3T} D_i\right)^2.
\]
Hence, combining with \eqref{eq:inequality-case-b} we obtain
\[
T^4 D_{18T}^5 \leq T^2D_{18T}^2\cdot 2^{15}\left(\frac{D_{\lceil \frac{T+1}{2} \rceil}}{D_{10T}}\right)^{2} \left(\sum_{i=0}^{\lfloor \frac{5T}{2}\rfloor} D_i \right)^2 \sum_{k=T}^{21T}G_k \leq 2^{17}\left(\sum_{i=0}^{3T} D_i \right)^4 \sum_{k=T}^{21T}G_k,
\]
as desired.
\end{proof}

\subsection{Proof of polynomial decay of $B_T$}

In this subsection, we prove the polynomial decay of $B_T$ and $D_T$.
We first state and prove the next purely analytical lemma. The constants have not been optimized. We let $C=1/\cos(\pi/8)$.
\begin{lemma}\label{lem: purely analytical}
Let $(D_k)_{k\geq 0}$ be a non-increasing sequence such that $D_k\in (0, C]$, and let $(G_k)_{k\geq 1}$ be such that $G_k>0$. Assume that for all $T\geq 1$
\begin{equation}\label{eq:inductive-inequality}
T^4D_{18T}^5\leq 2^{17} \left(\sum_{i=0}^{T} D_i\right)^4\cdot \sum_{k=T}^{21T}G_k,
\end{equation}
and also that $\sum_{k\geq T} G_k \leq 2D_{2T-1}$. Then for $\eps=10^{-10}$ one has $D_T\leq 100\cdot T^{-\eps}$ for all $T\geq 1$. 
\end{lemma}
\begin{proof}
We will show that $D_T\leq 22^{-1/10}\cdot 100\cdot T^{-\eps}$ for all $T$ of the form $22^{10^9\cdot s}$ with $s\geq 0$. This implies the desired result, since for $22^{10^9\cdot s}\leq k<22^{10^9\cdot (s+1)}$ we then have
\begin{equation}\label{eq: c to 2c}
D_k\leq D_T\leq 22^{-1/10}\cdot 100\cdot T^{-\varepsilon}= 100 \cdot 22^{-\varepsilon \cdot 10^9\cdot (s+1)}\leq 100\cdot k^{-\varepsilon},
\end{equation}
where we used that $(D_k)$ is non-increasing.

For $s=0$ there is nothing to prove, so let us assume that the desired inequality holds for some $S-1\geq 0$, and we will show it for $S$.

For $j\geq 0$, let $T_j:=22^{10^9(S-1)+j}$. Since $\sum_{k\geq T_0} G_k \leq 2D_{2T_0-1}\leq 2D_{T_0}$, for at least one $j_0\in\{0,1,\ldots,10^9-1\}$ we have 
\[
\sum_{k=T_{j_0}}^{21T_{j_0}}G_k \leq 10^{-9}\cdot 2D_{T_0}.
\]
Using the monotonicity of $(D_k)$, the inequality \eqref{eq:inductive-inequality} with $T=T_{j_0}$ together with the above bound, and finally the induction hypothesis, we deduce that for $K:=T_{10^9}$ we have
\[
D_{K}^5\leq D_{18 T_{j_0}}^5\leq 2^{18}\cdot 10^{-9}\cdot \left(T_{j_0}^{-1}\sum_{i=0}^{3T_{j_0}} D_i\right)^4\cdot D_{T_0}\leq 2^{17}\cdot 10^{-7}\cdot \left(T_{j_0}^{-1}\sum_{i=0}^{3T_{j_0}} D_i\right)^4\cdot T_0^{-\eps}.
\]
Since $D_i\leq 100 i^{-\eps}$ for all $1\leq i\leq K$ by our induction hypothesis and \eqref{eq: c to 2c}, 
we have that 
\[
\sum_{i=0}^{3T_{j_0}} D_i \leq C+100\sum_{i=1}^{3T_{j_0}} i^{-\eps} \leq (100+ C) \sum_{i=1}^{3T_{j_0}} i^{-\eps}\leq 102\int_0^{3T_{j_0}} x^{-\eps} \mathrm{d}x= 102 \frac{(3T_{j_0})^{1-\eps}}{1-\eps}. 
\]
Since $T_{j_0}\geq T_0$ we obtain
\[
D_K^5\leq 2^{17}\cdot 10^{-7}\cdot \left(102 \cdot T_{j_0}^{-1}\cdot \frac{(3T_{j_0})^{1-\eps}}{1-\eps}\right)^4\cdot T_0^{-\eps} \leq 2^{17}\cdot 10^{-7}\cdot \left(\frac{306}{1-\eps}\right)^4 T_0^{-5\eps}\leq 4\cdot 100^4 \cdot T_0^{-5\eps},
\]  
where we used that $2^{17}\cdot 10^{-7}\leq 4\cdot 3^4 \cdot 10^{-4}$ and $\frac{306}{1-\eps}\leq \frac{1000}{3}$.
Since $(K/T_0)^{5\eps}=\sqrt{22}$ and $4\cdot \sqrt{22} \cdot 100^4\leq \left(22^{-1/10}\cdot 100\right)^5$, 
it follows that 
\[
D_K^5\leq \left(22^{-1/10}\cdot 100\right)^5 K^{-5\varepsilon},
\]
from which the induction step follows. This completes the proof.
\end{proof}

As a corollary we obtain the desired polynomial decay of $B_T$ and $D_T$. %In what follows, we let $\varepsilon=10^{-10}$.

\begin{theorem}\label{thm: D_T polynomial decay}
Let $\eps:=10^{-10}$. For every $T\geq 1$ we have
\[
D_T\leq 100\cdot  T^{-\varepsilon}.
\]
\end{theorem}
\begin{proof}
Note that we have $\sum_{k\geq T} G_k \leq 2D_{2T-1}$ as a consequence of Lemma~\ref{lem: G sum}. This together with Lemma~\ref{lem: B-D ineq} and Corollary~\ref{cor: final ineq} implies that the assumptions of Lemma~\ref{lem: purely analytical} are satisfied and the result follows.
%.   Then the result follows from Lemma~\ref{lem: purely analytical}, noting that the assumptions of the lemma are satisfied by Lemmas~\ref{lem: G sum} and \ref{lem: B-D ineq}, and Corollary~\ref{cor: final ineq}.  
\end{proof}

\begin{proof}[Proof of Theorem~\ref{thm: polynomial decay}] Since we have $B_T\leq D_T$ by Lemma~\ref{lem: B-D ineq}, this follows immediately from Theorem~\ref{thm: D_T polynomial decay}.
\end{proof}

\section{Quantitative sub-ballisticity}  

In this section, we obtain a quantitative bound on the displacement of a typical self-avoiding walk.

We first need the following lemma. In what follows, we set $\varepsilon:=10^{-10}$.

\begin{lemma}\label{lem: long bridges}
There exists a constant $c>0$ such that
\[
\sum_{k\geq 5\cdot \eps^{-1}\cdot\tfrac{n}{\log(n+1)}} \sum_{\substack{\gamma\in B(\textup{Strip}_k):\\ \ell(\gamma)=n}} x_c^n\leq e^{-cn^{2/3}} 
\]
for every $n\geq 1$.
\end{lemma}
\begin{proof}
Consider a bridge $\gamma\in B(\textup{Strip}_k)$ for some $k\geq w:= 5\cdot \eps^{-1}\cdot\tfrac{n}{\log(n+1)}$ that has length $n$, and let $m:= \lfloor n^{1/3} \rfloor$. We first decompose $\gamma$ into a collection of bridges, with most of them having height $\frac{\sqrt{3}}{2}m$, and walks that connect consecutive bridges. 

To this end, for each $i=0,1,\ldots,m-1$ and $j=0,1,\ldots,s:=\lfloor w/m \rfloor-1$, we consider the horizontal lines $\ell_{i,j}=\{y=\frac{\sqrt{3}}{2}jm+i\}$. Let $t^{+}_{i,j}=\max\{t=0,\ldots,\ell(\gamma): \gamma_t\in \ell_{i,j}\}$ and $t^{-}_{i,j}=\min\{t=t^{+}_{i,j-1},\ldots,\ell(\gamma): \gamma_t\in \ell_{i,j}\}$, where we set for convenience $t^+_{i,-1}=0$. Note that the subwalks $\alpha(i):=(\gamma_0,\ldots,\gamma_{t^-_{i,0}})$, $b(i,j):=(\gamma_{t^{+}_{i,j}},\ldots,\gamma_{t^{-}_{i,j+1}})$, and $\tau(i):=(\gamma_{t^+_{i,s}},\ldots,\gamma_n)$  are all (up to translation) bridges; whilst the subwalks $c(i,j):=(\gamma_{t^{-}_{i,j}},\ldots,\gamma_{t^{+}_{i,j}})$ of $\gamma$ start and end on $\ell_{i,j}$. Write $I(i,j)$ for the number of times that $c(i,j)$ lies on $\ell_{i,j}$. By our choice of the range of $i$, the lines $\ell_{i,j}$ are disjoint, hence for some $i_0$ we have
\[
\sum_{j=0}^s I(i_0,j)\leq \frac{n+1}{m}.
\]

With this decomposition in hands, we claim that 
\begin{equation}\label{eq: length upper bound}
\sum_{k\geq 5\cdot \eps^{-1}\cdot\tfrac{n}{\log(n+1)}} \sum_{\substack{\gamma\in B(\textup{Strip}_k):\\ \ell(\gamma)=n}} x_c^n\leq \left(\sum_{i=1}^n B_i\right)^2 B_m^s (C+1)^{\frac{n+1}{m}},   
\end{equation}
where $C=2\sum_{k\geq 1} G_k<3$.
Indeed, the term $\left(\sum_{i=1}^n B_i\right)^2$ accounts for the possibilities for $\alpha(i_0)$ and $\tau(i_0)$. The term $B_m^s$
accounts for the $s$ bridges of height $\frac{\sqrt{3}}{2}m$. Finally, for the last term, observe that each $c(i_0,j)$ is the concatenation of (translations and rotations of) $U$-walks. Note that each time $c(i_0,j)$ visits the line $\ell_{i_0,j}$, there are two possibilities: either we attach a $U$-walk of positive length, which contributes at most $C$, or we do not attach any walk, and either we move on the next line $\ell_{i_0,j+1}$ (if such a line exists) or we stop, and this choice contributes at most $1$. Since the number of intersections with the lines $\ell_{i_0,j}$ is at most $\frac{n+1}{m}$, we can conclude that \eqref{eq: length upper bound} holds.

By Theorem~\ref{thm: polynomial decay}
\begin{equation*}
\begin{aligned}
B_m^s \leq (100\cdot m^{-\eps})^s = \exp\left(\eps\log{m}\cdot w/m + O(s)\right) \leq \exp\left(-\left(\frac{5}{3}+o(1)\right)n^{2/3}\right). 
\end{aligned}
\end{equation*}
Since 
\[
(C+1)^{\frac{n+1}{m}}\leq e^{(1+o(1))\log 4 \cdot n^{2/3}}
\]
and $\log 4<5/3$, it follows that the right-hand side of \eqref{eq: length upper bound} decays exponentially in $n^{2/3}$. To handle the case of small $n$ we note that 
\[
\sum_{k\geq 5\cdot \eps^{-1}\cdot\tfrac{n}{\log(n+1)}} \sum_{\substack{\gamma\in B(\textup{Strip}_k):\\ \ell(\gamma)=n}} x_c^n< b_n x_c^n\leq 1,
\]
where $b_n$ is the number of bridges of length $n$, and in the last inequality we used the standard inequality $b_n\leq \mu^n$.
This completes the proof.
\end{proof}

We are now ready to prove Theorem~\ref{thm: quantitative sub-ballisticity}.

\begin{proof}[Proof of Theorem~\ref{thm: quantitative sub-ballisticity}]
Let $x=(1/2,0)$ be the center of the hexagon on the right of $0$. Consider a SAW $\gamma$ of length $n$. We first project $\gamma-x$ along the lines $i\mathbb{R}$ and $e^{\pi i/6}\mathbb{R}$. Let $\textup{p}_{\theta}$ denote the projection along the vector $e^{i\theta}$, and note that either
\[
\max\{|\textup{p}_{\pi/2}(\gamma_k-x)|: 0\leq k\leq n\}\geq \cos\left(\frac{\pi}{6}\right)  \max\{\lVert \gamma_k -x \rVert : 0\leq k\leq n\}
\]
or
\[
\max\{|\textup{p}_{\pi/6}(\gamma_k-x)|: 0\leq k\leq n\}\geq \cos\left(\frac{\pi}{6}\right)  \max\{\lVert \gamma_k -x \rVert : 0\leq k\leq n\}.
\]
Moreover,
\[
\max\{\lVert \gamma_k -x \rVert : 0\leq k\leq n\}\geq \max\{\lVert \gamma_k \rVert : 0\leq k\leq n\} -1/2.
\]
Letting $a_n=5\cdot \eps^{-1}\cdot\tfrac{n}{\log(n+1)}$, by the $\pi/6$ rotational symmetry of the hexagonal lattice we have
\begin{equation*}
\begin{aligned}
\mathbb{P}_n[\max\{\lVert \gamma_k \rVert : 0\leq k\leq n\}\geq 2a_n]&\leq \mathbb{P}_n[\max\{\lVert \gamma_k-x \rVert : 0\leq k\leq n\}\geq 2a_n-1/2]    \\
&\leq 4\mathbb{P}_n\left[\max\{\textup{p}_{\pi/2}(\gamma_k): 0\leq k\leq n\}\geq \frac{\sqrt{3}}{2}\left(2a_n-1/2\right)\right] \\
&\leq 4\mathbb{P}_n\left[\max\{\textup{p}_{\pi/2}(\gamma_k): 0\leq k\leq n\}\geq \frac{\sqrt{3}}{2}a_n\right]
\end{aligned}
\end{equation*}
for every $n\geq 1$.

The Hammersley-Welsh unfolding operation that turns walks into bridges,
may be applied to walks satisfying 
$\max \{\textup{p}_{\pi/2}(\gamma_k): 0 \leq k \leq n\}\geq a_n$. 
Since $\max \{\textup{p}_{\pi/2}(\gamma_k)\}$
increases at each step
of the unfolding operation, and the resulting walk is a bridge, we can conclude that
\begin{equation*}
\begin{aligned}
\mathbb{P}_n\left[\max\{\textup{p}_{\pi/2}(\gamma_k): 0\leq k\leq n\}\geq \frac{\sqrt{3}}{2} a_n\right]&\leq \frac{1}{c_n}e^{C\sqrt{n}}x_c^{-n}\sum_{k\geq a_n} \sum_{\substack{\gamma\in B(\textup{Strip}_k):\\ \ell(\gamma)=n}} x_c^n \\
&\leq e^{C\sqrt{n}}\sum_{k\geq a_n} \sum_{\substack{\gamma\in B(\textup{Strip}_k):\\ \ell(\gamma)=n}} x_c^n
\end{aligned}
\end{equation*}
for some $C>0$, where we used the standard inequality $c_n\geq \mu^n=x_c^{-n}$.
The desired result follows now from Lemma~\ref{lem: long bridges}.
\end{proof}

\bibliographystyle{alpha}
\bibliography{surveybis.bib}

\end{document}